\documentclass[a4paper,10pt]{article}
\usepackage{amsmath}
\usepackage{amssymb,amscd,amsfonts,amsmath}
\usepackage{amsthm}

\newcommand{\seq}[1]{\left<#1\right>}
\newcommand{\norm}[1]{\left\Vert#1\right\Vert}
\newcommand{\mi}{\left\arrowvert}
\newcommand{\md}{\right\arrowvert}

\newtheorem{definition}{Definition}[section]
\newtheorem{theorem}{Theorem}[section]
\newtheorem{lemma}{Lemma}[section]
\newtheorem{corollary}[theorem]{Corollary}
\newtheorem{proposition}{Proposition}
\newtheorem{remark}[theorem]{Remark}

\newcommand{\supp}{\mbox{\rm\,supp\,}}
\newcommand{\btd}{\bigtriangledown}
\newcommand{\btu}{\bigtriangleup}
\newcommand{\dst}{\displaystyle}
\newcommand{\ba}{\begin{array}}
\newcommand{\ea}{\end{array}}
\newcommand{\refe}[1]{(\ref{#1})}
\newcommand{\half}{\mbox{\scriptsize $\frac{1}{2}$} }

\def\cstok#1{\leavevmode\thinspace\hbox{\vrule\vtop{\vbox{\hrule\kern1pt
\hbox{\vphantom{\tt/}\thinspace{\tt#1}\thinspace}}
\kern1pt\hrule}\vrule}\thinspace}

\title{A high order $q$-Difference equation for $q$-Hahn multiple orthogonal polynomials}
\author{Jorge Arves\'u \thanks{Electronic address:
jarvesu@math.uc3m.es} \\
Department of Mathematics, Universidad Carlos III de Madrid,\\ Avenida de la Universidad, 30, 28911, Legan\'es, Spain;\\
Chiara Esposito \thanks{Electronic address:
esposito@math.ku.dk}\\
Department of Mathematics, Copenhagen University,\\ 
Universitetsparken 5, DK-2100 Copenhagen, Denmark\\}

\begin{document}

\maketitle

\begin{abstract}
A high order linear $q$-difference equation with polynomial
coefficients having $q$-Hahn multiple orthogonal polynomials as
eigenfunctions is given. The order of the equation is related to the number of orthogonality conditions that these polynomials satisfy. 
Some limiting situations when $q\to1$ are studied. Indeed, the difference equation
for Hahn multiple orthogonal polynomials given in \cite{Lee} is corrected and obtained as a limiting case.
\end{abstract}

\section{Introduction}

The relevance of the special functions and classical orthogonal polynomials as eigenfunctions of a second order differential equation is a well established fact \cite{nubook}. For instance, many fundamental problems of quantum mechanics as the harmonic oscillator, the solution of the Schr\"{o}dinger, Dirac and Klein-Gordon equations for a Coulomb potential, the motion of a particle in homogeneous electric or magnetic field lead to the study of the eigenfunctions for the generalized equation of hypergeometric type (see \cite{nubook}). Recall that the spherical and cylinder (Bessel) functions  -perhaps the most popular special functions- as well as the classical orthogonal polynomials are particular solutions of this equation (see also equation \refe{class-hyp-eq} below).

Recently, the appearance of new special functions, namely, the multiple orthogonal polynomials also having the property of being eigenfunctions of a differential equation have attracted the interest of several researchers -see \cite{jc-van_assche}, and \cite{Lee} for a discrete case-. In an effort to consider more general situations, the notion of $q$-Hahn multiple orthogonal polynomials was introduced in \cite{arvesu-q-multi-Hahn}. Furthermore, a particular situation in which a third order $q$-difference equation having the aforementioned multiple orthogonal polynomials as eigenfunctions was considered. Here we complete the study initiated in \cite{arvesu-q-multi-Hahn} by obtaining a high order $q$-difference equation having $q$-Hahn multiple orthogonal polynomials as eigenfunctions. This $q$-difference equation gives a relation for the polynomial with a given degree evaluated at non-uniformed distributed points $x(s)$, $x(s+1),\dots,x(s+r+1)$; where $x(s)$ denotes the $q$-exponential lattice \refe{q-expo} and $r$, the number of orthogonality conditions \refe{ortho-condition-multiple-hahn}. In addition, other multiple orthogonal polynomial families like some studied in \cite{arvesu-coussement-vanassche,artikel}
as well as their corresponding differential and difference equations (see \cite{jc-van_assche,Lee}) can be obtained as a limiting case. 

The contents of this paper are as follows. In Section \ref{pre}, some well-known results on classical and multiple orthogonality are summarized. In Section \ref{main-result}
we give an explicit expression for the multiple orthogonal polynomials studied here as well as a high order $q$-difference equation that these polynomials satisfy. Indeed, the notion of lowering and raising operators as well as the auxiliary lemma \ref{Lee's-lemma} play an essential role in the accomplishment of the $q$-difference equation (theorem \ref{thm2}). As corollary we get some previous results \cite{arvesu-q-multi-Hahn,Lee}. Lastly, Section \ref{limiting} comprises some limiting cases known in the literature as multiple Jacobi and Hahn polynomials, respectively \cite{arvesu-coussement-vanassche,jc-van_assche,artikel}. We also show how these families of polynomials can be obtained from the $q$-Hahn multiple orthogonal polynomials based on the Rodrigues-type formula.

\section{Preliminary notions\label{pre}}
One of the most remarkable features of the classical orthogonal polynomials on the real line, namely, Jacobi, Laguerre, and Hermite polynomials is that they 
are eigenfunctions of the second order differential operator \cite{abraste, Chihara}
\begin{align}
\dst \left(a_{2}(x)\dfrac{d^{2}}{dx^{2}}+a_{1}(x)\dfrac{d}{dx}+\lambda\mathcal{I}\right)
y(x)=0,
\label{class-hyp-eq}
\end{align}
where $\deg a_{2}(x)\leq
2$, $\deg a_{1}(x)=1$, $\lambda$ is a constant independent on $x$ and $\mathcal{I}$ denotes the identity operator. The above equation is known as hy\-per\-geo\-me\-tric equation. An exhaustive study of their solutions in term of the polynomial coefficients $a_2(x)$ and $a_1(x)$ for each classical family can be found in \cite{nubook}.
% From the perspective of differential operators, polynomials orthogonal on the unit circle behave very differently with respect to the real case.

The notion of orthogonality for these classical families is assumed to be with
respect to a Borel measure $\mu$ on the real line $\mathbb{R}$ (with infinitely many points of increase) supported on a subset
$\Omega\subset\mathbb{R}$, where $\Omega=(a,b)$, with $|a|<\infty$, is the smallest interval on the real line that contains
$\supp \mu$. This orthogonality concept can be briefly described as follows: If $\int_{\Omega}p(x)d\mu(x)$ converges for every polynomial
$p$, then one can define an inner product
$$
\dst\seq{p,q}=\int_{\Omega}p(t)q(t)d\mu(t),
$$
where $p,q$ are polynomials. For such an inner product, which we will say to
be standard, a sequence of polynomials $(p_{n})_{n\geq0}$ is said to be
orthogonal with respect to the above inner product if
\begin{enumerate}
\item[(i)] $\deg p_{n}(x)=n$,
\item[(ii)] $\seq{p_{n},p_{m}}=\delta_{n,m}\norm{p_n}^{2}$, $m,n\in\mathbb{N}$, and $\delta_{m,n}$ denotes the Kronecker delta function.
\end{enumerate}

A difference analog of the equation \refe{class-hyp-eq} on the lattice $x(s)$, $\{x(s)\mapsto\mathbb{R}^{+}:\,a\leq s\leq
 b-1\}$, is the hypergeometric-type difference equation \cite{Nikiforov}
\begin{equation}
\ba{c} \displaystyle
\left(\sigma(s) \frac{\btu}{\btu x(s-\half)}  \frac{\btd }{\btd x(s)}
+ \tau(s) \frac{\btu }{\btu x(s)} + \lambda \mathcal{I}\right)y(s) =0, \\
y(s)=y(x(s)),\quad \sigma(s)=a_2(x(s)) - \half a_1(x(s)) \btu x(s-\half),
 \quad
\tau(s)=a_1(x(s)),
 \ea
\label{eqdif}
\end{equation}
being $\bigtriangledown y(s)=\bigtriangleup y(s-1)$. Here $\btu$ and $\btd$ are the forward and backward difference operators, respectively. Observe that $\bigtriangleup f(s)=f(s+1)-f(s)$. 

Analogously, the classical orthogonal polynomials of a discrete variable can be obtained as the corresponding eigenfunctions of \refe{eqdif}. Indeed,
the polynomial family that verifies \refe{eqdif} can be orthogonalized by constructing a Sturm-Liouville problem with orthogonalizing weight $\omega(s)$ as the solution of a Pearson-type difference equation (see \cite[pp. 70-72]{Nikiforov}). Indeed, 
the following orthogonality property 
\begin{equation}
\ba{c} \displaystyle \sum_{s = a }^{b-1} P_n(x(s))
P_m(x(s))\omega(s) \bigtriangleup x(s-\half) =
\delta_{n,m}||P_n||^2, \ea \label{norm}
\end{equation}
yields under the additional (boundary) conditions $\dst\left.\sigma(s) \omega(s) x^{k}(s-\half) \right|_{s=a,b} = 0$,
$ k= 0,1,\dots$ From now on we will
denote any polynomial $P_n(x(s))$ simply as $P_n(s)$.

In \cite{arvesu} was studied the $q$-Hahn orthogonal polynomials $P_n^{\alpha,\beta}(s)$, $n=0,1,\dots$, with parameters $\alpha,\beta>-1$, on the lattice 
\begin{equation}
x(s)=\frac{q^{s}-1}{q-1},\quad q\not=1.
\label{q-expo}
\end{equation}
In particular, the orthogonalizing weight and the coefficients of the second order difference
equation \refe{eqdif} are as follows:
\begin{equation}
\omega(s)= \begin{cases}\dst
v_{q}^{\alpha,\beta,N}(s)=\frac{\tilde{\Gamma}_q(s+\beta+1)
 \tilde{\Gamma}_q(N+\alpha-s+1)}{q^{-\frac{\alpha+\beta}{2}s}\tilde{\Gamma}_q(s+1)\tilde{\Gamma}_q(N+1-s)} & s=0,1,\dots,N\in\mathbb{N},\\
0,&\quad \mbox{otherwise},
\end{cases}\label{v-function}
\end{equation}
where $\mathbb{N}$ denotes the set of natural numbers, $\tilde{\Gamma}(s)=q^{-\frac{(s-1)(s-2)}{4}}f(s;q)$ if $0<q<1$, or 
$\tilde{\Gamma}(s)=f(s;q^{-1})$ if $q>1$, being $f(s;q)=(1-q)^{1-s}
\frac{\prod_{k\geq0} (1-q^{k+1})  }
{\prod_{k\geq0} (1-q^{s+k})}$, and
\begin{align}
\dst\sigma(s)= - q^{-\frac{N+\alpha}{2}}x(s)^{2}+q^{-\frac{1}{2}}
[N+\alpha]_{q}x(s),\quad\mbox{where}\quad
[x]_q
%\stackrel{\mbox{\tiny def}}{=}
=\frac{q^{\frac{x}2} -
q^{-\frac{x}2}} {q^{\frac{1}2} -
q^{-\frac{1}2}},\nonumber\\
\dst\tau(s)= -q^{\frac{\beta+2-N}{2}}[\alpha+\beta+2]_{q}x(s) +
q^{\alpha+\beta+1}[\beta+1]_{q}[N-1]_{q},\label{coef-sigma-tau}\\
\dst \lambda_n = q^{\frac{\beta+2-N}{2}} [n]_q [n+\alpha+\beta+1]_{q}.\nonumber
\end{align}

For the classical orthogonal polynomials there is a characterization due to
Sonin (and independently obtained by W. Hahn) in the sense that they are the
families of orthogonal polynomials such that the sequence of their first
derivatives constitutes again a sequence of orthogonal polynomials with a shifted set of parameters. For classical orthogonal polynomials of a discrete variable there is a similar characterization. When one takes the forward difference of the classical orthogonal polynomials of a discrete variable, one can show that these polynomials are again orthogonal polynomials of the same family, but with a different set of parameters. As such, the $\frac{\btu}{\btu x(s-\half)}$ operator acts as a {\it lowering operator}
on these families of polynomials. As an example for the $q$-Hahn polynomials
\begin{equation}
\frac{\btu}{\btu x(s-\half)} P_n^{\alpha,\beta,N}(s) =  q^{-\frac{n}{2}}[n]_{q}P_{n-1}^{\alpha+1,\beta+1,N-1}(s).
\label{lowering}
\end{equation}
Also for the classical orthogonal polynomials of a discrete variable we have
{\it raising operators} which can also be found using the orthogonality property \refe{norm} and summation 
by parts. Hence,
\begin{equation}
\begin{array}{c} \dst\left(
\frac{1}{v_{q}^{\alpha-1,\beta-1,N+1}(s)}\frac{ \btd }{\btd x
(s)}
v_{q}^{\alpha,\beta,N}(s)\right)P^{\alpha,\beta,N}_{n}(s)\\
\dst=-\dfrac{
\left[n+\alpha+\beta\right]_{q}}{q^{\frac{N+n+\beta}{2}}}
P^{\alpha-1,\beta-1,N+1}_{n+1}(s).
\end{array}
\label{raising1}
\end{equation}
By combining the lowering and raising operators \refe{lowering}-\refe{raising1} one can obtain the equation \refe{eqdif} for the $q$-Hahn orthogonal polynomials with the above polynomial coefficients \refe{coef-sigma-tau}. This fact is a general feature of the classical orthogonal polynomials of a discrete variable, i.e., the corresponding combination of lowering a raising operators leads to the second order difference equation
\refe{eqdif}.

In section \ref{main-result} we will generalize the above procedure for a non-standard orthogonality. Indeed, a new situation appears, namely, a third order $q$-difference equation having $q$-orthogonal polynomials as eigenfunctions. This situation keeps certain similarity with the previous accomplishment of the hypergeometric equation via raising and lowering operators; however two orthogonality conditions must be considered instead (see \cite{arvesu-q-multi-Hahn}). Hence, we need to deal with the
concept of multiple orthogonal polynomial. This notion appears naturally in simultaneous Pad\'e
approximation -which is often known as Hermite-Pad\'e approximation- in order
to get simultaneous rational approximants to a vector function. As we will see below they can be interpreted as an extension of ordinary orthogonal polynomials.

Let $\mu_{1},\mu_{2},\dots,\mu_{r}$ $(r\geq2)$ be
Borel measures on $\mathbb{R}$ with infinitely many points of increase such that $\supp\mu_{i}
\subset{{\Omega}_{i}}$, $i=1,\dots,r$, where ${\Omega}_{i}=(a_{i},b_{i})$, with $|a_{i}|<\infty$, is the smallest interval on the real line that contains
$\supp\mu_i$. The Cauchy transforms of the corresponding measures 
%Some authors make a distingtion and use the denomination of Markov functions when the measure has compact support, and Stieltjes otherwise.
$$
\dst \hat{\mu}_{i}(z)=\int_{{\Omega}_{i}}\frac{d\mu_{i}(x)}{z-x},
\quad z\notin\Omega_{i}\quad i=1,\dots,r,
$$
also known as Markov (or Stieltjes) functions, can be simultaneously approximated by rational functions with
prescribed order near infinity. For studying this problem \cite{Nikishin} one needs to introduce a multi-index $\vec{n}=(n_{1},
n_{2},\dots,n_{r})$ of nonnegative integers, and finds a polynomial $P_{\vec{n}}(z)\not\equiv0$ of
degree at most $|\vec{n}|=n_1+\cdots+n_r$, such that the expressions
\begin{equation*}
\ba{c}
\dst P_{\vec{n}}(z)\hat{\mu}_{i}(z)-Q_{\vec{n},i}(z)=\frac{\zeta_{i}}{z^{n_{i}+1}}+\cdots=\mathcal{O}(z^{-n_{i}-1}),\quad
i=1,\dots,r,
\ea
\label{f1}
\end{equation*}
hold. Observe that the polynomial $P_{\vec{n}}(z)$ is the common denominator
of the simultaneous rational approximants $Q_{\vec{n},i}(z)/P_{\vec{n}}(z)$, to the Markov (Stieltjes) functions 
$\hat{\mu}_{i}(z)$, $i=1,2,\dots,r$.
% \label{rational-function-pi}
% \eq

Indeed, $P_{\vec{n}}(z)$ is a type II
multiple orthogonal polynomial of degree
$\leq|\vec{n}|$ defined by the following orthogonality conditions (see \cite{arvesu-coussement-vanassche,Nikishin,artikel})
\begin{equation}
\int_{\Omega_i}P_{\vec{n}}(x)x^k\ d\mu_i(x) =  0,\quad k=0,1,\ldots,n_i-1,\quad i=1,\dots,r.
\label{stelseltype2}
\end{equation}
These conditions \refe{stelseltype2}
give a linear system of $|\vec{n}|$ homogeneous equations for the
$|\vec{n}|+1$ unknown coefficients of $P_{\vec{n}}(z)$.  If the
multi-index $\vec{n}$ is {\em normal} \cite{Nikishin} the solution
is a unique polynomial $P_{\vec{n}}(z)$ (up to a multiplicative
factor) of degree exactly $|\vec{n}|$. In this situation
throughout the paper we consider always monic multiple orthogonal
polynomials.

If the measures in \refe{stelseltype2} are positive discrete
measures on $\mathbb{R}$, i.e.,
$$
\mu_i=\sum_{k=0}^{N_i}\omega_{i,k}\delta_{x_{i,k}},\qquad
\omega_{i,k}>0,\ x_{i,k}\in\mathbb{R},\ \ N_i\in\mathbb{N}\cup
\{+\infty\},\ \  i=1,\ldots,r,
$$
where $x_{i_{1},k}\not=x_{i_{2},k}$, $k=0,\ldots,N_i$, whenever
$i_1\not=i_2$, the corresponding polynomial solution is then a
{\em discrete multiple orthogonal polynomial} $P_{\vec{n}}(z)$ \cite{arvesu-coussement-vanassche}.
Here we have that $\supp\mu_i$ is the closure of
$\{x_{i,k}\}_{k=0}^{N_i}$ and that $\Omega_i$ is the smallest
closed interval on $\mathbb{R}$ that contains
$\{x_{i,k}\}_{k=0}^{N_i}$.  If the above system of measures forms
an {\em AT system} then every multi-index is
normal (see \cite{Nikishin} for a detailed explanation on the concept). 

In \cite{arvesu-coussement-vanassche} for several AT systems of
measures was studied the corresponding discrete multiple
orthogonal polynomials of type II on the linear lattice $x(s)=s$
(those of Charlier, Kravchuk, Meixner of first and second kind,
and Hahn). Even more, it was obtained rising operators and then
the Rodrigues-type formula.

Here we will deal with an AT system formed by different discrete measures supported on the same interval. This situation was analyzed in \cite{arvesu-q-multi-Hahn} for $\vec{n}=(n_1,\ldots,n_r)\in\mathbb{N}^{r}$ and
the set of parameters $N$, $\alpha_{0}$ and $\vec{\alpha}=(\alpha_1,\ldots,\alpha_r)$. Indeed, the orthogonality conditions \refe{stelseltype2} were considered with respect to the following positive discrete
measures on $\mathbb{R}$
\begin{equation}
\mu_{i}=\sum_{s=0}^{N}v^{\alpha_{i},\alpha_{0},N}_{q}(k)\bigtriangleup
x(k-\half)\delta(k-s),\quad i=1,\dots,r,
\label{q-Hahn-measures}
\end{equation}
where $v^{\alpha_{i},\alpha_{0},N}_{q}(k)$ is defined in \refe{v-function} and 
$\alpha_{0},\alpha_{i}>-1$, $\alpha_i - \alpha_j \notin
\{0,1,\ldots,N-1\}$ when $i \neq j$.

\begin{definition}\label{def-1} A polynomial $P_{\vec{n}}^{\vec{\alpha},\alpha_{0},N}(s)$ that verifies the orthogonality conditions \refe{stelseltype2} with respect to the measures \refe{q-Hahn-measures} is said to be the type-II $q$-Hahn multiple orthogonal polynomial of a multi-index
$\vec{n}\in\mathbb{N}^{r}$, i.e., the conditions
\begin{equation}\begin{array}{c}
\dst\sum_{s=0}^{N}P_{\vec{n}}^{\vec{\alpha},\alpha_{0},N}(s)(s)_{q}^{[k]}v^{\alpha_{i},\alpha_{0},N}_{q}(s)
\bigtriangledown x(s+\half)=0,\quad k\leq n_i-1,\quad
i=1,\dots,r,
\end{array}\label{ortho-condition-multiple-hahn}
\end{equation}
hold, where the symbol $(s)_{q}^{[k]}$ denotes the following polynomial of degree at most $k$ in the variable $x(s)$
\begin{equation*}
(s)_{q}^{[k]}=\prod_{j=0}^{k-1}\frac{q^{s-j}-1}{q-1}=x(s)x(s-1)\cdots
x(s-k+1).
\label{canonical-base}
\end{equation*}
\end{definition}

Notice that $(s)_{q}^{[k]}$ is a polynomial of degree $k$ in $x(s)$, and the orthogonality conditions \refe{stelseltype2} have been written more conveniently as \refe{ortho-condition-multiple-hahn}. When $r=1$ definition \ref{def-1} gives the standard orthogonality \refe{norm}. In addition,
$$
    \{1,\left(s\right)_q^{[1]},\ldots ,\left(s\right)_q^{[k]}\},
$$
constitutes a basis of the linear space of polynomials of degree at most $k$ in $x(s)$. 
Indeed, we consider this basis as the canonical one. When $q$ goes to $1$, $(s)_{q}^{[k]}$ converges to
$(-1)^{k}(-s)_k$, where $(s)_k$ denotes the usual Pochhammer symbol
$(s)_k=s(s+1)\cdots(s+k-1)$, $(s)_{0}=1$.

An important feature for $P_{\vec{n}}^{\vec{\alpha},\alpha_{0},N}(s)$ is the existence of raising operators \cite{arvesu-q-multi-Hahn}
\begin{align}
D^{\alpha_{i},\alpha_{0},N}\,P^{\vec{\alpha},\alpha_{0},N}_{\vec{n}}(s)
&\dst=-q^{-\frac{N+|\vec{n}|+\alpha_0}{2}}
\left[|\vec{n}|+\alpha_{i}+\alpha_{0}\right]_{q}
P^{\vec{\alpha}-\vec{e}_{i},\alpha_{0}-1,N+1}_{\vec{n}+\vec{e}_{i}}(s),
\label{raising}\\
 D^{\alpha_{i},\alpha_{0},N}&\stackrel{\mbox{\tiny
def}}{=}\left(\frac{1}{v_{q}^{\alpha_{i}-1,\alpha_{0}-1,N+1}(s)}\nabla
v_{q}^{\alpha_{i},\alpha_{0},N}(s)\right),\nonumber
\end{align}
where the multi-index $\vec{e}_i$ denotes the standard $r$ dimensional
unit vector with the $i$th entry equals $1$ and $0$ otherwise, and $\dst\nabla\stackrel{\mbox{\tiny def}}{=}\frac{\btd}{\btd x(s)}$. Indeed,
\begin{align}
D^{\alpha_i,\alpha_0,N}f(s)=A^{\alpha_0,\alpha_i,N}(s;q) \btd f(s)+B^{\alpha_0,\alpha_i,N}(s;q)f(s),
\label{raising-a-b}
\end{align}
where
\begin{align}
A^{\alpha_0,\alpha_i,N}(s;q)&={q^{-\frac{(\alpha_i+\alpha_0)}{2}}}[s]_q[N+\alpha_0-s+1]_q\label{raising-a}\\
B^{\alpha_0,\alpha_i,N}(s;q)&=\left([s+\alpha_{i}]_{q}[N-s+1]_{q}
-\frac{[s]_{q}[N+\alpha_{0}-s+1]_{q}}{q^{\frac{\alpha_{i}+\alpha_{0}}{2}}}\right)\nonumber\\
&=\left([N+1]_q [\alpha_i]_q-x(s)[\alpha_i+\alpha_0]_qq^{-\frac{(\alpha_{0}+N)}{2}}\right).
\label{raising-f}
\end{align}

\section{High order $q$-difference equation \label{main-result}}

The next theorem \ref{thm2} extends the result of \cite{arvesu-q-multi-Hahn} concerning the $q$-difference equation. We will combine the lowering and raising operators to get a
$(r+1)$-order $q$-difference equation. As a corollary, an explicit difference equation for the Hahn multiple orthogonal polynomials 
is given (see \cite{Lee}).

An explicit expression for the type-II monic $q$-Hahn multiple orthogonal polynomials can be obtained as a direct
consequence of a summation by parts and formula \refe{raising}. 
\begin{proposition}\label{q-Hahn-Rodrigues}
The following finite-difference analog of the Rodrigues formula
\begin{equation}
  \begin{array}{c} \dst P^{\vec{\alpha}, 
  \alpha_{0},N}_{\vec{n}}(s)=\frac{\dst(-1)^{\left|\vec{n}\right|}
  q^{\frac{(N+\alpha_{0})|\vec{n}|+ \prod_{i=1}^{r}n_i}{2}+\sum_{i=1}^r\binom{n_{i}}{2}}}
{ \prod_{k=1}^{r}(|\vec{n}|+\alpha_{0}+\alpha_{k}+1|q)_{n_k}
  }
  \dst\frac{q^{-\frac{\alpha_{0}}{2}s}
  \tilde{\Gamma}_{q}(s+1)\tilde{\Gamma}_{q}(N-s+1)}{\tilde{\Gamma}_{q}(\alpha_{0}+N-s+1)}\\
\dst
  \left(\prod_{i=1}^{r}\frac{q^{-\frac{\alpha_{i}}{2}s}}{\tilde{\Gamma}_{q}(\alpha_{i}+s+1)}\nabla^{n_i}
  \frac{\tilde{\Gamma}_{q}(\alpha_{i}+n_{i}+s+1)}{q^{-\frac{\alpha_{i} 
  +n_{i}}{2}s}}
  \right)
  \frac{q^{\frac{\alpha_{0}+|\vec{n}|}{2}s}\tilde{\Gamma}_{q} 
  (\alpha_{0}+N-s+1)}{\tilde{\Gamma}_{q}(s+1)\tilde{\Gamma}_{q}(N-| 
  \vec{n}|-s+1)},
\label{rodrigues-q-hahn-multiple}
  \end{array}
\end{equation}
holds, where $(a|q)_k=\prod_{m=0}^{k-1}[a+m]_{q}=\tilde{\Gamma}_{q}(a+k)/\tilde{\Gamma}_{q}(a)$ is the $q$-analogue of the Pochhammer symbol. 
\end{proposition}
Notice that the above Rodrigues-type formula characterizes the type-II $q$-Hahn multiple orthogonal polynomials in terms of a finite-difference property.

\begin{proof} Replacing $(s)_{q}^{[k]}$ in \refe{ortho-condition-multiple-hahn} by the following finite-difference expression
$$
(s)_{q}^{[k]}=\frac{ q^{\frac{(k-2)}{2}} }{[k+1]_{q}}
\nabla (s+1)_{q}^{[k+1]},
$$
the orthogonality conditions can be
written in a more convenient way as follows
\begin{equation}\begin{array}{c} \dst\sum_{s=0}^{N}P^{\vec{\alpha}, 
  \alpha_{0},N}_{\vec{n}}(s)\nabla (s+1)_{q}^{[k+1]}v^{\alpha_{i},\alpha_{0},N}_{q}(s)\bigtriangledown x(s+\half)=0,\\
k=0,1,\dots,n_i-1,\quad
i=1,2,\dots,r.\end{array}\label{eqv-otho-cond-hahn}
\end{equation}
From here, using the summation by parts one gets \refe{raising}. Recursively using this raising operator gives the
Rodrigues-type formula \refe{rodrigues-q-hahn-multiple}.
\end{proof}

In the next theorem we will need the following auxiliary lemma. 
\begin{lemma}\label{Lee's-lemma} Let $A$ be the following $r$-dimensional matrix 
\begin{align*}
A&=\left(\begin{array}{cccc} \frac{1}{[n_1]_{q}} & \frac{1}{[n_1+\alpha_1-\alpha_2]_{q}} & \cdots & \frac{1}{[n_1+\alpha_1-\alpha_r]_{q}}\\
\frac{1}{[n_2+\alpha_2-\alpha_1]_{q}} & \frac{1}{[n_2]_{q}} & \cdots & \frac{1}{[n_2+\alpha_2-\alpha_r]_{q}}\\
\vdots & \vdots & \ddots & \vdots\\
\frac{1}{[n_r+\alpha_r-\alpha_1]_{q}} & \frac{1}{[n_r+\alpha_r-\alpha_2]_{q}} & \cdots & \frac{1}{[n_r]_{q}}
\end{array}
\right)\\
&=(a_{i,j})_{i,j=1}^{r},\quad a_{i,j}=\dfrac{1}{[n_i+\alpha_{i}-\alpha_{j}]_{q}},
\end{align*}
then the determinant of $A$ is
\begin{equation}
\det A=\dfrac{\dst\prod_{k=1}^{r-1}\prod_{l=1}^{r}[\alpha_{k}-\alpha_{l}]_{q}[n_l-n_k+\alpha_{l}-\alpha_{k}]_{q}}
{\dst\prod_{k=1}^{r}\prod_{l=1}^{r}[n_l+\alpha_{l}-\alpha_{k}]_{q}}.
\label{det-A}
\end{equation}
\end{lemma}
Here for proving \refe{det-A} we will follow the operations indicated in \cite[Lemma 2.8, p. 18]{Lee}.
\begin{proof} Let us proceed by column and row operations on the matrix $A$. Observe that, for $k=1,\dots,r$ and $i=2,\dots,r$ the following relation
\begin{align}
a_{k,i}-a_{k,1}&=\tilde{\lambda}_{i,1}a_{k,i}a_{k,1}q^{-\frac{n_{k}+\alpha_{k}}{2}}\left(q^{\frac{\alpha_{i}+\alpha_{1}}{2}}+q^{\alpha_{k}+n_{k}}\right),\quad 
\tilde{\lambda}_{i,1}=\dfrac{[\alpha_{i}-\alpha_{1}]_{q}}{
q^{\frac{\alpha_{i}}{2}}+q^{\frac{\alpha_{1}}{2}} },
\label{column-op}
\end{align}
yields.

Therefore, based on \refe{column-op} if $A_k$ denotes the $k$th column of $A$ $(k=1,\dots,r)$  one gets
\begin{align*}
\det A&=\det(A_1,A_2-A_1,\dots,A_r-A_1)\\
&=\left(\prod_{k=1}^{r} a_{k,1}\right)\left(\prod_{i=2}^{r}\tilde{\lambda}_{i,1}\right)\left(\begin{array}{cccc} 1& \tilde{a}_{1,2} &\cdots&\tilde{a}_{1,r}\\
1 & \tilde{a}_{2,2}&\cdots&\tilde{a}_{2,r}\\
\vdots & \vdots & \ddots & \vdots\\
1& \tilde{a}_{r,2} &\cdots&\tilde{a}_{r,r}
\end{array}
\right),
\end{align*}
where $\tilde{a}_{k,i}=a_{k,i}q^{-\frac{n_{k}+\alpha_{k}}{2}}\left(q^{\frac{\alpha_{i}+\alpha_{1}}{2}}+q^{\alpha_{k}+n_{k}}\right)$. Now, if one substracts the first row from the other ones, and takes into account that
$$
\tilde{a}_{k,i}-\tilde{a}_{1,i}=a_{k,i}a_{1,i}\mu_{k,1}\left(q^{\frac{\alpha_i}{2}}+q^{\frac{\alpha_1}{2}}\right),
\quad\mu_{k,1}=[n_1-n_k+\alpha_1-\alpha_k]_{q}\quad i,k=2,\dots,r,
$$
then
$$
\det A=\left(\prod_{k=1}^{r} a_{k,1}\right)\left(\prod_{j=2}^{r} a_{1,j}\right)\left(\prod_{i=2}^{r}\lambda_{i,1}\mu_{i,1}\right)\left(\begin{array}{cccc} a_{2,2}& a_{2,3} &\cdots& a_{2,r}\\
a_{3,2} & a_{3,3}&\cdots& a_{3,r}\\
\vdots & \vdots & \ddots & \vdots\\
a_{r,2}& a_{r,3} &\cdots& a_{r,r}
\end{array}
\right),
$$
where $\lambda_{i,1}=[\alpha_{i}-\alpha_{1}]_{q}$.

Finally, repeating the previous column and row operations -but on lower dimensional matrices- the expression \refe{det-A} can be inductively proved.
\end{proof}

\begin{theorem}\label{thm2} The type-II monic $q$-Hahn multiple orthogonal polynomial  $P^{\vec{\alpha},\alpha_{0},N}_{\vec{n}}(s)$ is an eigenfunction of the following
$(r+1)$-order $q$-difference equation
  \begin{equation}\begin{array}{c} \dst
  \left(\prod_{i=1}^{r}D^{\alpha_{i}+1,\alpha_{0}+1,N-1}\right) 
  \frac{\btu}{\btu
  x(s)}P^{\vec{\alpha},\alpha_{0},N}_{\vec{n}}(s)=-q^{- 
  \frac{(N+|\vec{n}|+\alpha_{0}-1)}{2}}\\
  \dst
  \left(\sum_{i=1}^{r}\xi_{i}[|\vec{n}|+\alpha_{0}+\alpha_{i}+1]_{q} D^{\alpha_{i}+1,\alpha_{0}+1,N-1} 
  \right)
  P^{\vec{\alpha},\alpha_{0},N}_{\vec{n}}(s),
\end{array}
  \label{q-diff-eq-multiple-Hahn}
  \end{equation}
where
  \begin{align}
\nonumber \xi_{i}=&\left(\sum_{k=1}^{r}\dfrac{(-1)^{k+l}[n_k+\alpha_k+\alpha_0+1]_{q}\prod_{i=1,i\not=l}^{r}[n_k+\alpha_k-\alpha_i]_{q}}
{\prod_{i=1,i\not=k}^{r-1}[n_i+\alpha_i-n_k-\alpha_k]_{q}\prod_{j=k+1}^{r}[n_k+\alpha_k-n_j-\alpha_j]_{q}}\right)\\
&\dfrac{q^{\frac{|\vec{n}|-1}{2}}\prod_{k=1}^{r-1}\prod_{l=1}^{r}[\alpha_{k}-\alpha_{l}]_{q}[n_l-n_k+\alpha_{l}-\alpha_{k}]_{q}}
{\prod_{k=1}^{r}\prod_{l=1}^{r}[n_l+\alpha_{l}-\alpha_{k}]_{q}},\quad i=1,\dots,r.
\label{xi-constant}
  \end{align}

  \end{theorem}
\begin{proof} Taking into account the expressions \refe{raising}-\refe{raising-f}, the $q$-Hahn multiple orthogonal polynomial can be expressed in term of the raising operator as follows
  $$\begin{array}{c}
  P^{\vec{\alpha},\alpha_{0},N}_{\vec{n}}(s)\dst=\frac{-q^{\frac{N+\left|\vec{n}\right|+\alpha_0-1}{2}}}
  {\left[\left|\vec{n}\right|+\alpha_{i}+\alpha_{0}\right]_{q}}\\
  \left(
B^{\alpha_{0}+1,\alpha_{i}+1,N-1}(s;q)
  \mathcal{I}+A^{\alpha_{0}+1,\alpha_{i}+1,N-1}(s;q)\btd
  \right)
P^{\vec{\alpha}+\vec{e}_{i},\alpha_{0}+1,N-1}_{\vec{n}- 
  \vec{e}_{i}}(s),
  \end{array}
  $$
 where $i=1,2,\dots, r$. Hence, for $k\leq r$, one gets the following relation
  \begin{eqnarray}\label{intermediate}
  \nonumber\dst\sum_{s=0}^{N}P^{\vec{\alpha}, 
  \alpha_{0},N}_{\vec{n}}(s)(s)_{q}^{[n_{k}-1]}
  v^{\alpha_{k}+1,\alpha_{0},N}_{q}(s)\btu x(s-\half)\\
\dst=\frac{-q^{\frac{N+\left|\vec{n}\right|+\alpha_0-1}{2}}}
  {\left[\left|\vec{n}\right|+\alpha_{i}+\alpha_{0}\right]_{q}}\sum_{s=0}^{N}(s)_{q}^{[n_{k}-1]}
  v^{\alpha_{k}+1,\alpha_{0},N}_{q}(s)\btu x(s-\half)\\
 \nonumber \left(
B^{\alpha_{0}+1,\alpha_{i}+1,N-1}(s;q)
  \mathcal{I}%\right.\\\dst\left.
+A^{\alpha_{0}+1,\alpha_{i}+1,N-1}(s;q)\btd
  \right)P^{\vec{\alpha}+\vec{e}_{i},\alpha_{0}+1,N-1}_{\vec{n}- 
  \vec{e}_{i}}(s).
  \end{eqnarray} 
Now,  transforming \refe{intermediate} by doing a suitable combination of the orthogonalizing weight $v^{\alpha_{k}+1,\alpha_{0},N}_{q}(s)$ with the terms $A^{\alpha_{0}+1,\alpha_{i}+1,N-1}(s;q)$ and $B^{\alpha_{0}+1,\alpha_{i}+1,N-1}(s;q)$, and summing by parts
  \begin{equation}\label{lemma-1-final}\begin{array}{c}
  \dst\sum_{s=0}^{N}P^{\vec{\alpha}, 
  \alpha_{0},N}_{\vec{n}}(s)(s)^{[n_k-1]}_{q}v^{\alpha_{k} 
  +1,\alpha_{0},N}_{q}
  (s)\btu x(s-\half)=\frac{q^{\theta}[n_{k}+\alpha_{k}-\alpha_{i}]_{q}} 
  {[\left|\vec{n}\right|+\alpha_{i}+\alpha_{0}+1]_{q}}\\
  \dst\times\sum_{s=0}^{N}P^{\vec{\alpha}+\vec{e}_{i}, 
  \alpha_{0}+1,N-1}_{\vec{n}-\vec{e}_{i}}(s)
  (s)^{[n_k-1]}_{q}v^{\alpha_{k}+1,\alpha_{0}+1,N-1}_{q}(s)\btu
  x(s-\half),
  \end{array}
  \end{equation}
  yields, where $\theta=\frac{N+|\vec{n}|+n_{k}+\alpha_{k}+\alpha_{0}}{2}$. Here, aimed to shift the parameters $\alpha_0$ and $N$ in the orthogonalizing weight $v^{\alpha_{k}+1,\alpha_{0},N}_{q}(s)$ we have used the relations
\begin{align*}
\dfrac{v^{\alpha_{k}+1,\alpha_{0}+1,N-1}_{q}(s)}{v^{\alpha_{k}+1,\alpha_{0},N}_{q}(s)}&= q^{\frac{s}{2}}[N-s]_{q},\\
\dfrac{v^{\alpha_{k}+1,\alpha_{0}+1,N-1}_{q}(s)}{v^{\alpha_{k}+1,\alpha_{0},N}_{q}(s+1)}
&\dst=\dfrac{(s+1)_{q}^{[n_{k}-1]}\left[N+\alpha_{0}-s\right]_{q}}{q^{\frac{\alpha_{k}+\alpha_{0}-1}{2}}
  (s)_{q}^{[n_{k}-2]}[\alpha_{k}+s+2]_{q}}.
\end{align*}

By using recursively relation \refe{lemma-1-final} one gets
 \begin{equation}
  \begin{array}{c}
  \dst\sum_{s=0}^{N}P^{\vec{\alpha}+\vec{e}_{l}, 
  \alpha_{0}+1,N-1}_{\vec{n}-\vec{e}_{l}}(s)(s)_{q}^{[n_{k}-1]}
  v^{\alpha_{k}+1,\alpha_{0}+1,N-1}_{q}(s)\btu x(s-\half)\\
  \dst=\tilde{a}_{k,l}\sum_{s=0}^{N}
  P^{\vec{\alpha}+\vec{e},\alpha_{0}+r,N-r}_{\vec{n}-\vec{e}}(s) 
  (s)_{q}^{[n_{k}-1]}
  v^{\alpha_{k}+1,\alpha_{0}+r,N-r}_{q}(s)\btu x(s-\half),
  \end{array}
  \label{exp-lemma-1}
  \end{equation}
where $\vec{e}=\sum_{i=1}^{r}\vec{e}_{i}$ and
  $$
  \dst \tilde{a}_{k,l}=  
  \frac{[\left|\vec{n}\right|+\alpha_{0}+\alpha_{l}+1]_{q}}{q^{-(r-1)\left(\theta-1\right)-\frac{1}{2}}[\alpha_{k}- 
  \alpha_{l}+n_{k}]_{q}}
\left(\prod_{j=1}^{r}\frac{[\alpha_{k}-\alpha_{j}+n_{k}]_{q}}
  {[\left|\vec{n}\right|+\alpha_{0}+\alpha_{j}+1]_{q}}\right),\,\,k,l=1,\dots,r.
  $$

Now, based on \refe{column-op} we will find that there exists constants $\left\lbrace \xi_{l}\right\rbrace _{l=1}^r$ (not all zero) such that the relation
\begin{equation}
\Delta P^{\vec{\alpha},\alpha_{0},N}_{\vec{n}}(s)=
  \sum_{l=1}^r\xi_{l}P^{\vec{\alpha}+\vec{e_l},\alpha_{0}+1,N-1}_{\vec{n}-\vec{e_l}}(s),\quad \Delta=\dst\frac{\btu}{\btu
  x(s)},
\label{lowering-important}
\end{equation}
is valid. Thus, for finding explicitly $\xi_1,\dots,\xi_r$ one takes into account \refe{exp-lemma-1} and \refe{lowering-important} to get
 \begin{equation}\begin{array}{l}
  \dst\sum_{s=0}^{N}\left(\Delta P^{\vec{\alpha},\alpha_{0},N}_{\vec{n}}(s)\right) 
  (s)_{q}^{[n_{k}-1]}
  v^{\alpha_{k}+1,\alpha_{0}+1,N-1}_{q}(s)\btu x(s-\half)\quad\quad\\
  \dst=\left(\sum_{l=1}^{r}\xi_{l}\tilde{a}_{k,l}\right)\sum_{s=0}^{N}
  P^{\vec{\alpha}+\vec{e},\alpha_{0}+r,N-r}_{\vec{n}-\vec{e}}(s) 
  (s)_{q}^{[n_{k}-1]}
  v^{\alpha_{k}+1,\alpha_{0}+r,N-r}_{q}(s)\btu x(s-\half).
  \end{array}
  \label{lemma-2-equation-2}
  \end{equation}
Now, the left hand side of this equation can be easily transformed by means of a summation by parts and orthogonality relation
  \refe{ortho-condition-multiple-hahn}. Indeed,
  $$
  \begin{array}{c}
 \dst\sum_{s=0}^{N}\left(\Delta P^{\vec{\alpha},\alpha_{0},N}_{\vec{n}}(s)\right) 
  (s)_{q}^{[n_{k}-1]}
  v^{\alpha_{k}+1,\alpha_{0}+1,N-1}_{q}(s)\btu x(s-\half)\quad\quad\\
  \dst
  =\dst\frac{[n_{k}+\alpha_{0}+\alpha_{k}+1]_{q}}{q^{\theta-\frac{|\vec{n}|-1}{2}}}
  \sum_{s=0}^{N}
  P^{\vec{\alpha},\alpha_{0},N}_{\vec{n}}(s) 
  (s)_{q}^{[n_{k}-1]}
  v^{\alpha_{k}+1,\alpha_{0},N}_{q}(s)\btu x(s-\half).
  \end{array}
  $$
  Based on \refe{lemma-1-final} and \refe{exp-lemma-1} the right hand side of this expression transforms into the equation 
  \begin{equation}
  \begin{array}{c}
  \dst\sum_{s=0}^{N}\left(\Delta P^{\vec{\alpha},\alpha_{0},N}_{\vec{n}}(s)\right) 
  (s)_{q}^{[n_{k}-1]}
  v^{\alpha_{k}+1,\alpha_{0}+1,N-1}_{q}(s)\btu x(s-\half)\\
  =\dst \tilde{b}_{k}\sum_{s=0}^{N}
  P^{{\vec{\alpha}+\vec{e},\alpha_{0}+r,N-r}}_{\vec{n}-\vec{e}}(s) 
  (s)_{q}^{[n_{k}-1]}
  v^{\alpha_{k}+1,\alpha_{0}+r,N-r}_{q}(s)\btu x(s-\half),
  \end{array}
  \label{lemma-2-equation-3}
  \end{equation}
where 
 $$
  \tilde{b}_{k}=q^{\frac{|\vec{n}|}{2}+(r-1)(\theta-1)}[n_{k}+\alpha_{k}+\alpha_{0}+1]_{q}
  \prod_{i=1}^r\frac{[n_{k}+\alpha_{k}- 
  \alpha_{i}]_{q}}
  {[|\vec{n}|+\alpha_{0}+\alpha_{i}+1]_{q}}.
  $$
  From equations \refe{lemma-2-equation-2} and
  \refe{lemma-2-equation-3} leads the following linear system of equation for the unknown coefficients $\xi_1,\dots,\xi_r$,
  \begin{equation}
  \dst b_{k}=\sum_{l=1}^r\xi_{l} s_{k,l},\quad k=1,\dots,r,\quad\Longleftrightarrow\quad S \xi=b,\quad \xi=(\xi_1,\dots,\xi_r),
  \label{system-lambdas}
  \end{equation}
where the entries of the vector $b$ and matrix $S$ are as follows
$$
b_{k}=q^{\frac{|\vec{n}|-1}{2}}[n_k+\alpha_k+\alpha_0+1]_{q},\quad s_{k,l}=a_{k,l}[|\vec{n}|+\alpha_{0}+\alpha_l+1]_{q}.
$$
By the Cramer’s rule, the above system \refe{system-lambdas} has a unique solution if and only if
the determinant of $S$ is different from zero. Observe that $S=A\cdot D$, where $D$ denotes the diagonal matrix
$$
D=(d_{k,l})_{k,l=1}^{r},\quad d_{k,l}=[|\vec{n}|+\alpha_{0}+\alpha_l+1]_{q}\delta_{k,l}.
$$
Thus, from lemma \ref{Lee's-lemma}, formula \refe{det-A}, one gets
$$
\det S=\det (A\cdot D)=\left(\prod_{i=1}^{r}[|\vec{n}|+\alpha_{0}+\alpha_i+1]_{q}\right)\det A\not=0.
$$
Accordingly, if $C_{i,j}$ is the cofactor of the entry $s_{i,j}$, and $S_j (b)$ denotes the matrix obtained from $S$ replacing its $j$th column by $b$, then
$$
\xi_{l}=\dfrac{\det S_{l}(b) }{\det S},\quad l=1,\dots,r,
$$
where
\begin{align*}
\det S_{l}(b) &=\sum_{k=1}^{r}b_{k}C_{k,l}\\
&=\sum_{k=1}^{r}b_{k}\dfrac{(-1)^{k+l}\prod_{i=1,i\not=l}^{r}[n_k+\alpha_k-\alpha_i]_{q}}
{\prod_{i=1,i\not=k}^{r-1}[n_i+\alpha_i-n_k-\alpha_k]_{q}\prod_{j=k+1}^{r}[n_k+\alpha_k-n_j-\alpha_j]_{q}}.
\end{align*}
Consequently, expression \refe{xi-constant} yields.
 
Finally, applying the following product of $r$ operators $\left(\prod_{i=1}^{r}D^{\alpha_{i}+1,\alpha_{0}+1,N-1}\right)$ on both sides of the equation \refe{lowering-important} and
considering that these raising operators are commuting, the expression \refe{q-diff-eq-multiple-Hahn} holds. %\qed
\end{proof}

In particular when $r=2$, the $q$-Hahn multiple orthogonal polynomial verify a third order $q$-difference equation (see \cite[theorem 2.1, pp. 7-8]{arvesu-q-multi-Hahn} and next corollary).

\begin{corollary}\label{corq3}
The type-II monic $q$-Hahn multiple orthogonal polynomial  $P^{\alpha_1,\alpha_2,\alpha_{0},N}_{n_1,n_2}(s)$
verify the following third order $q$-difference equation
\begin{equation}\label{q3-order}
a_{4}(s)\Delta\nabla^{2}y+a_{3}(s)\Delta\nabla
y +a_{2}(s)\Delta y+a_{1}(s)\nabla y
+a_{0}(s) y=0, 
\end{equation}
where 
\begin{align*}
a_4(s)&=q^{2s-3}A^{\alpha_0+1,\alpha_1+1,N-1}(s;q)A^{\alpha_0+1,\alpha_2+1,N-1}(s-1;q),\\
a_3(s)&=A^{\alpha_0+1,\alpha_1+1,N-1}(s;q)
\left(\frac{B^{\alpha_0+1,\alpha_2+1,N-1}(s-1;q)}{ q^{1-s}}+\btd \frac{A^{\alpha_0+1,\alpha_2+1,N-1}(s;q)}{q^{1-s}}\right)\\
&+q^{s-1}B^{\alpha_0+1,\alpha_1+1,N-1}(s;q)A^{\alpha_0+1,\alpha_2+1,N-1}(s;q),\\
a_2(s)&=B^{\alpha_0+1,\alpha_1+1,N-1}(s;q)B^{\alpha_0+1,\alpha_2+1,N-1}(s;q)\\
&+A^{\alpha_0+1,\alpha_1+1,N-1}(s;q)\btd B^{\alpha_0+1,\alpha_2+1,N-1}(s;q),\\
a_1(s)&=q^{-\frac{N+|\vec{n}|+\alpha_0-2s+3}{2}}\left(\frac{\xi_1 A^{\alpha_0+1,\alpha_2+1,N-1}(s;q)}{\left[ |\vec{n}|+\alpha_1+\alpha_0+1 \right]_q^{-1}}
+\frac{\xi_2 A^{\alpha_0+1,\alpha_1+1,N-1}(s;q)}{\left[ |\vec{n}|+\alpha_2+\alpha_0+1 \right]_q^{-1}}\right), \\
a_0(s)&=q^{-\frac{N+|\vec{n}|+\alpha_0+1}{2}}\left( \frac{\xi_1 B^{\alpha_0+1,\alpha_2+1,N-1}(s;q)}{\left[|\vec{n}|+\alpha_1+\alpha_0+1 \right]_q^{-1}}
+\frac{\xi_2 B^{\alpha_0+1,\alpha_1+1,N-1}(s;q)}{\left[ |\vec{n}|+\alpha_2+\alpha_0+1 \right]_q^{-1}}\right),
\end{align*}
and $\xi_1=q^{\frac{n_1+n_2-1}{2}}\frac{[n_1]_q[n_2+\alpha_2-\alpha_1]_q}{[\alpha_2-\alpha_1]_q}$, $\xi_2=q^{\frac{n_1+n_2-1}{2}}\frac{[n_2]_q[n_1+\alpha_1-\alpha_2]_q}{[\alpha_1-\alpha_2]_q}$.
\end{corollary}
The above $q$-difference equation can be considered as an extension of the hypergeometric-type difference equation \refe{eqdif}. Again here the
combination of lowering and
rising operators \refe{raising-a-b} was the key fact to obtain a third order difference
equations having $q$-Hahn multiple orthogonal polynomials as
eigenfunctions \cite{arvesu-q-multi-Hahn}. 
\begin{remark} Notice that from the $q$-difference operator \refe{q-diff-eq-multiple-Hahn} can be obtained the difference equation studied in \cite{Lee} for the Hahn multiple orthogonal polynomials since when $q$ goes to $1$ the lattice $x(s)$ transforms into a linear one $s$. In particular, for the above third order $q$-difference equation \refe{q3-order} one gets the same type of third order difference equation with the following polynomial coefficients
$$
\begin{array}{rl}
a_4(s)&=s(s-1)(\alpha_{0}+N-s+1)(\alpha_{0}+N-s+2),\\
a_3(s)&=s(\alpha_{0}+N-s+1)\left[2\alpha_0+\alpha_2 +4+(\alpha_1+\alpha_2+3)N\right.\\
&\left.-(2\alpha_0+\alpha_1+\alpha_2+6)s \right],\\
a_2(s)&=[(\alpha_1+1) N-(\alpha_0+\alpha_1+2) s] [(\alpha_2+1) N-(\alpha_0+\alpha_2+2) s]\\
&-(\alpha_0+\alpha_2+2) (\alpha_0+N-s+1) s,\\
%a_1(s)&=[(\alpha_0+N+1)(n_1+n_2) +n_1\alpha_1+n_2\alpha_2-n_1n_2] (\alpha_0+N-s+1) s,\\
a_1(s)&=[n_1 (\alpha_0 + \alpha_1 + N+1) + (\alpha_0 + \alpha_2 + N+1) n_2-n_1n_2] (\alpha_0+N-s+1) s,\\
a_0(s)&=N [(\alpha_2+1) (\alpha_0 + \alpha_1 + N+1) n_1 + (\alpha_1+1) (\alpha_0 + \alpha_2 + 
      N+1) n_2 \\
&+ (\alpha_0 + N) n_1 n_2]
-[(\alpha_0 + \alpha_2+2) (\alpha_0 + \alpha_1 + N+1) n_1\\ 
&+ (\alpha_0 + \alpha_1+2) (\alpha_0 + \alpha_2 + 
    N+1) n_2 + (N-1) n_1 n_2]s.
\end{array}
$$
\end{remark}

\section{Limiting transitions \label{limiting}}

Observe that when $q$ goes to $1$ the expression \refe{rodrigues-q-hahn-multiple} transforms into the Ro\-dri\-gues-type formula for Hahn multiple orthogonal polynomials in the linear lattice $x(s)=s$ (see \cite{arvesu-coussement-vanassche})
\begin{align}
&\dst H^{\vec{\alpha},\alpha_{0},N}_{\vec{n}}(s)=\frac{(-1)^{|\vec{n}|}}
{\prod_{k=1}^{r}\left(\mi\vec{n}\md+\alpha_{0}
+\alpha_{k}+1\right)_{n_{k}}}
\frac{\Gamma(s+1)\Gamma(N-s+1)}{\Gamma(\alpha_{0}+N-s+1)}\nonumber\\
&\left(\prod_{k=1}^{r}\frac{1}{\Gamma(\alpha_{i}+s+1)}\nabla^{n_{i}}
\Gamma(\alpha_{i}+n_{i}+s+1)\right)\frac{\Gamma(\alpha_{0}+N-s+1)}
{\Gamma(s+1)\Gamma(N-|\vec{n}|-s+1)}. \label{hahn-rodrigues}
\end{align}

Now, based on this Rodrigues-type formula another limiting transition
between discrete and continuous multiple orthogonal polynomials can also be obtained. The basic tool for establishing this limiting transition is the usual approximation of derivatives by means of finite-differences.

Suppose that $f(s)$ is a function defined on an interval of the real line,
which contains the linear lattice $\{s_{i}\}_{i=0}^{N}$.
Furthermore, $f(s)$ possesses first derivative on
$\{s_{i}\}_{i=0}^{N}$, and second derivative for every
$\chi_{i}\in(s_{i}-h,s_{i})$, $i=1,\dots,N$. Thus,
\begin{equation*}\ba{rl} \dst\nabla
f(s_{i})&\dst=\frac{f(s_{i})-f(s_{i}-h)}{h}=f'(s_{i})-\frac{h}{2}f''(\chi_{i})\\
&=f'(s_{i})+\mathcal{O}(h), \ea\label{1-derivative}
\end{equation*}
yields. In general, if $f(s)$ has $n$ derivatives at points $s_i$ and $2n$
derivatives for any $\chi_{i}$ one gets
\begin{equation}
\dst\nabla^{n}f(s_{i})=f^{(n)}(s_{i})+\mathcal{O}(h^{n}).
\label{k-derivative}
\end{equation}

Notice that the change of variable $s=Nx$, transforms the interval
$[0,N]$ into $[0,1]$, which is precisely the interval of orthogonality for the multiple Jacobi polynomials studied in \cite{Nikishin,pinheiro,artikel}. 
The simultaneous orthogonality conditions were considered with respect to an AT-system of Jacobi weights on [0, 1] with different singularities at $0$ and the same
singularity at $1$. Below we will show that when $N$ tends to
infinity, i.e., the step $h=\btu x=1/N$ in the new
variable $x$ tends to $0$, the Hahn multiple orthogonal
polynomials \refe{hahn-rodrigues} will tend to the aforementioned monic multiple Jacobi polynomials 
$P^{(\alpha_{0},\vec{\alpha})}_{\vec{n}}(x)$ (see below the explicit expression \refe{jacobi-rodrigues}).
\begin{proposition}
The following limiting relation is valid:
\begin{equation}
\dst\lim_{N\to\infty}N^{-|\vec{n}|}H^{\vec{\alpha},\alpha_{0},N}_{\vec{n}}(Nx)=
P^{\vec{\alpha},\alpha_{0}}_{\vec{n}}(x).
\label{Hahn-Jacobi-limit}
\end{equation}
\end{proposition}
\begin{proof} For simplicity let us consider the multi-index $\vec{n}=(n_{1},n_{2})$ as well as the interval $[0,N-1]$ as the support of the orthogonality measures \refe{q-Hahn-measures}. The proof for $r>2$, i.e., $\vec{n}=(n_{1},n_{2},\dots,n_{r})$ follows the same steps described below. 

Firstly, let us show that for $N$ large enough the term $\Gamma(N-Nx)/\Gamma(\alpha_{0}+N-Nx)$ contained in \refe{hahn-rodrigues} behaves like 
\begin{equation}
(1-x)^{-\alpha_{0}}N^{-\alpha_{0}}.
\label{1st-estimation}
\end{equation}
This estimation follows immediately from the well known asymptotic relation for the gamma-function \cite{nubook} 
\begin{equation}
\frac{\Gamma(z+a)}{\Gamma(z)}=z^{a}
\left[1+\mathcal{O}\left(\frac{1}{z^{2}}\right)\right],\quad
\left|\arg z\right|\leq\pi-\delta,\quad\delta>0.
\label{well-known-limit}
\end{equation} 

Second, for $N$ large enough we will prove that the expression
\begin{equation}
\dst N^{-|\vec{n}|}\left(\prod_{k=1}^{r}\frac{\Gamma(Nx+1)}{\Gamma(\alpha_{i}+Nx+1)}\nabla^{n_{i}}
\frac{\Gamma(\alpha_{i}+n_{i}+Nx+1)}{\Gamma(Nx+1)}\right)\frac{\Gamma(N+\alpha_{0}-Nx)}{\Gamma(N-|\vec{n}|-Nx)}, 
\label{limit-D-Hahn}
\end{equation}
behaves like
\begin{equation}
N^{\alpha_{0}}\left(\prod_{i=1}^{2}x^{-\alpha_{i}}\frac{d^{n_{i}}}{d
x^{n_{i}}}x^{\alpha_{i}+n_{i}}\right)(1-x)^{\alpha_{0}+|\vec{n}|}.
\label{behaves}
\end{equation}
Indeed, using the relation for the $n$th backward difference
\begin{align}
\dst \btd^{n} y(x)=\sum_{k=0}^{n}\frac{(-1)^{k}n!}{k!(n-k)!}y(x-k)
=\sum_{k=0}^{n}\frac{(-n)_{k}}{k!}y(x-k), 
\label{n-th-backward}
\end{align}
the expression \refe{limit-D-Hahn} becomes
\begin{equation}
\ba{c}%\lim_{N\to\infty}
\dst\frac{N^{-|\vec{n}|}\Gamma(Nx+1)}{\Gamma(Nx+\alpha_{1}+1)}
\sum_{j=0}^{n_{2}}
\sum_{k=0}^{n_{1}}\frac{(-n_{2})_{j}}{j!}\frac{(-n_{1})_{k}}{k!}
\frac{\Gamma(N(x-k)+\alpha_{1}+n_{1}+1)}{\Gamma(N(x-k)+1)}\\
\dst\frac{\Gamma(N(x-k)+1)\Gamma(N(x-j-k)+\alpha_{2}+n_{2}+1)
\Gamma(N+\alpha_{0}-N(x-j-k))}
{\Gamma(N(x-k)+\alpha_{2}+1)\Gamma(N(x-j-k)+1)N-|\vec{n}|-N(x-j-k))}.
\ea \label{limit-D-Hahn-1}
\end{equation}
Hence, taking into account \refe {well-known-limit} the above expression \refe{limit-D-Hahn-1} can be rewritten as
\begin{equation}\ba{c}
\dst (Nx)^{-\alpha_{1}}\sum_{j=0}^{n_{2}}
\sum_{k=0}^{n_{1}}\frac{(-n_{2})_{j}}{j!}\frac{(-n_{1})_{k}}{k!}
(N(x-k))^{\alpha_{1}+n_{1}}(N(x-k))^{-\alpha_{2}}\\
\dst (N(s-j-k))^{\alpha_{2}+n_{2}}
(N-N(x-j-k))^{\alpha_{0}+|\vec{n}|}+\mathcal{O}\left(\frac{1}{N^{2}}\right).
\ea \label{limit-D-Hahn-2}\end{equation}
Using again \refe{n-th-backward}, but this time to express
\refe{limit-D-Hahn-2} in terms of the backward difference operators, i.e.,
\begin{equation}
\dst N^{\alpha_{0}}\left(
\prod_{i=1}^{2}x^{-\alpha_{i}}\nabla^{n_{i}}x^{\alpha_{i}+n_{i}}\right)(1-x)^{\alpha_{0}+|\vec{n}|}
+\mathcal{O}\left(\frac{1}{N^{2}}\right).
\label{limit-D-Hahn-3}
\end{equation}
Now, according to \refe{k-derivative} one can express the backward difference operators involved in
\refe{limit-D-Hahn-3} in terms of the ordinary de\-ri\-va\-ti\-ves. Thus, for $N$ large enough one verifies that indeed \refe{limit-D-Hahn} behaves like \refe{behaves}. 

Finally, from \refe{hahn-rodrigues} and the above estimations \refe{1st-estimation} and  \refe{behaves}, the proposition holds.
\end{proof}

Recall that the multiple Jacobi polynomials are given explicitly by the Rodrigues-type formula \cite{artikel}
\begin{equation}
\begin{array}{c}
\dst P^{\vec{\alpha},\alpha_{0}}_{\vec{n}}(x)=\dst\frac{(-1)^{|\vec{n}|}(1-x)^{-\alpha_{0}}}
{\prod_{i=1}^{r}\left(\mi\vec{n}\md+\alpha_{0}
+\alpha_{i}+1\right)_{n_{i}}}
\left(\prod_{i=1}^{r}x^{-\alpha_{i}}\frac{d^{n_{i}}}{
dx^{n_{i}}}x^{\alpha_{i}+n_{i}}\right)(1-x)^{\alpha_{0}+|\vec{n}|}. \end{array} \label{jacobi-rodrigues}
\end{equation}

In fact, 
$P^{\vec{\alpha},\alpha_{0}}_{\vec{n}}(x)$ verifies the following
orthogonality conditions
$$%\ba{c}
\dst \int_{0}^{1} P^{\vec{\alpha},\alpha_{0}}_{\vec{n}}(x)
x^{\alpha_{i}}(1-x)^{\alpha_{0}}x^{k}dx=0,%\\
\quad k=0,1,\dots,n_{i}-1,\quad i=1,2,\dots,r.%\ea
$$

Regarding the differential equation that these polynomials satisfy we refer to \cite{jc-van_assche}. Observe that this case constitutes a special limiting case of \refe{q-diff-eq-multiple-Hahn}.

\section{Conclusion and future directions}

To the best of our knowledge there is not in the literature any other high order $q$-difference equation different from \refe{q-diff-eq-multiple-Hahn} having multiple orthogonal polynomials -with $q$-discrete orthogonality- as eigenfunctions. Furthermore, the $q$-difference equation obtained here is quite versatile since other difference and differential equations can be simply obtained as a limiting case. For instance, the difference and differential equations for the multiple Hahn and Jacobi polynomials, respectively are examples of these limiting cases. Also the well known hypergeometric-type difference equation for Hahn polynomials as well as the hypergeometric equation for Jacobi polynomials are particular cases when $r=1$. However, more general situations demand our attention. Firstly, the multiple orthogonal polynomials on the lattice $x(s)=c_1q^s+c_2q^{-s}+c_3$, where $c_1$, $c_2$ and $c_3$ are constants independent on $s$ must be considered in the same fashion that here. Finally, more general systems of measures such that under certain restrictions on their elements (measures) one can recover the $q$-difference equation \refe{q-diff-eq-multiple-Hahn} must be analyzed. In this direction, the Askey-Wilson multiple orthogonal polynomials could be an interesting challenge to be considered. 

In closing, this paper outlines the important points and techniques to be followed in future investigations aimed to deduce the high order difference equations for $q$-Charlier, $q$-Kravchuk and $q$-Meixner multiple orthogonal polynomials.


\begin{thebibliography}{15}

\bibitem{abraste} M. Abramowitz, I. A. Stegun,
\textit{Handbook of Mathematical Functions}, National Bureau of
Standards, Washington, 1964; Dover, New York, 1972.

\bibitem{arvesu} J. Arves\'u, Quantum algebras $su_{q}(2)$ and $su_{q}(1,1)$ associated with
certain $q$-Hahn polynomials: A revisited approach, Electronic
Transaction on Numerical Analysis 24 (2006), 24--44.

\bibitem{arvesu-q-multi-Hahn} J. Arves\'u, On some properties of $q$-Hahn multiple orthogonal
polynomials, J. Comput.\ Appl.\ Math.\ (2009), doi:10.1016/j.cam.2009.02.062

\bibitem{arvesu-coussement-vanassche} J. Arves\'u, J. Coussement,
W. Van Assche Some discrete multiple orthogonal polynomials, J.
Comput.\ Appl.\ Math.\ 153 (2003), 19--45.

\bibitem{jc-van_assche} Differential equations for multiple orthogonal polynomials with respect to classical weights:
raising and lowering operators, J. Phys. A: Math. Gen. 39 (2006), 3311--3318.

\bibitem{Chihara} T. S. Chihara,
\textit{An Introduction to Orthogonal Polynomials}, Gordon and
Breach, New York, 1978.

\bibitem{Lee} D.W. Lee, Difference equations for discrete multiple
orthogonal polynomials, J. Approx. Theory (2007), doi:
10.1016\slash j.jat.2007.06.002.

\bibitem{Nikiforov} A. F. Nikiforov, S. K. Suslov, V. B. Uvarov,
\textit{Classical Orthogonal Polynomials of a Discrete Variable},
Springer-Verlag, Berlin, 1991.

\bibitem{nubook}  A. F. Nikiforov, V. B. Uvarov,
\textit{Special Functions of Mathematical Physics},
Birkh\"auser Verlag, Basel, 1988.

\bibitem{Nikishin} E. M. Nikishin, V. N. Sorokin,
\textit{Rational Approximations and Orthogonality}, Translations
of Mathematical Monographs, vol. 92, Amer.\ Math.\ Soc.,
Providence, RI, 1991.

\bibitem{pinheiro} L. R. Pi\~neiro, \textit{On simultaneous approximations for a
collection of Markov functions}, Vestnik Moskov. Univ., Ser. I (1987), 67-70;
English transl. in Moscow Univ. Math. Bull. 42 (1987), 52-55.

\bibitem{artikel} W. Van Assche, E. Coussement,
\textit{Some classical multiple orthogonal polynomials}, J.
Comput.\ Appl.\ Math.\ 127 (2001),  317--347.
\end{thebibliography}
\end{document}